\documentclass[twocolumn,secnumarabic,amssymb,nobibnotes,aps,prd]{revtex4-1}

\usepackage{hyperref}
\usepackage{graphicx}% Include figure files
\usepackage{dcolumn}% Align table columns on decimal point
\usepackage{bm}% bold math

\usepackage{xcolor}
\usepackage{amsthm,mathtools}
\usepackage{upgreek}	% for upright Greek letters, e.g. Dirac Delta: \updelta

\newcommand{\e}{\mathrm{e}}
\newcommand{\R}{\mathbb{R}}
\newcommand{\C}{\mathbb{C}}
\newcommand{\N}{\mathbb{N}}
\newcommand{\diff}{\mathrm{d}}	% Differential operator
\newcommand{\im}{\mathrm{i}}	% Imaginary unit

\newcommand{\F}{\mathcal{F}}

\let\Re\relax

\DeclareMathOperator{\Re}{Re}

\DeclareMathOperator{\supp}{supp}	% support of a function
\newcommand{\deriv}[3][]{\frac{\diff^{#1}#2}{\diff #3^{#1}}}
\newcommand{\pderiv}[3][]{\frac{\partial^{#1}#2}{\partial #3^{#1}}}

\newtheoremstyle{theorem}{}{}{\itshape}{}{\bfseries}{}{\newline}{}

\theoremstyle{theorem}
\newtheorem{theorem}{Theorem}%[section]
\newtheorem{lemma}[theorem]{Lemma}

\newtheoremstyle{example}{}{}{}{}{\bfseries}{}{\newline}{}

\theoremstyle{definition}
\newtheorem{example}[theorem]{Example}

\begin{document}

\title{A Spectral View on Slow Invariant Manifolds \texorpdfstring{\\}{} 
in Complex-Time Dynamical Systems}

\author{J\"orn~Dietrich, Dirk~Lebiedz}
\affiliation{%
Institut f\"ur Numerische Mathematik,  
Universit\"at Ulm, 
Helmholtzstra\ss{}e 20, 
89081 Ulm,
Germany
}%
\received{December 3, 2019}

 % Function theory, analysis

\date{\today}

\begin{abstract}
\noindent
	Many real-analytic flows, e.g.\ in chemical kinetics, share a multiple time scale spectral structure. 
	The trajectories of the corresponding dynamical systems are observed to bundle near so-called slow invariant manifolds (SIMs), which are usually addressed in a singular perturbation context.
	This work exploits the analytic structure of the involved vector fields and presents observations that connect one dimensional slow invariant manifolds to the imaginary-time spectral structure of Riemann surfaces in analytic continuations of dynamical systems to complex time.
\end{abstract}

\maketitle

\section{Introduction}
\noindent
We introduce analytic continuation of smooth autonomous dynamical systems from real to complex-time domain in order to study slow invariant manifolds (SIM) in terms of geometry and spectral properties of holomorphic curves, respectively embedded Riemann surfaces.\\ 
Imaginary time has been discussed in various contexts of modern physical theories such as special and general relativity, quantum mechanics and quantum field theory. 
The time-dependent Schr\"odinger equation is formally of parabolic heat-equation type in imaginary time. 
For time-independent potential energy in the quantum Hamiltonian operator (with the consequence of separability of the time-dependent Schr\"odinger equation), time-differentiation $\im\frac{\diff}{\diff t}$ can be interpreted as an energy operator. 
Wick rotation~\cite{Wick1954} is a technique bridging statistical and quantum mechanics by identifying the thermal Boltzmann factor $\frac{1}{kT}$ with imaginary time $\frac{2\pi\im t}{h}$. 
In special relativity replacing real by imaginary time makes the Minkowski metric Euclidean. 
For the classical mechanical model of a frictionless nonlinear pendulum in Hamiltonian formulation, the introduction of imaginary time opens relations between periodic solutions and the theory of elliptic integrals and functions (see~\cite{Ochs2011,Romero2018} for recent accounts). \\
Obviously, in several contexts there is a relation between physical properties, corresponding modeling techniques, and the formal introduction of imaginary time. 
However, to the best of our knowledge, analytic continuation to complex time resulting in holomorphic dynamical systems with Riemann surfaces instead of trajectories as solutions while embedding the original real dynamical system has not been used systematically to study slow invariant manifolds. 
In particular, we show here that a complex analysis view on flows of dynamical systems has the potential to shed new light on special phase space structures, the slow invariant manifolds. 
We exploit the global spectral structure of imaginary time trajectories and present results for three well known example models which have been analyzed using Fenichel's singular perturbation approach. The rich mathematics of Riemann surfaces provides significant potential for deeper studies.

\section{Invariant manifolds of slow motion}
\noindent
In systems with multiple time scales one typically observes bundling behavior of trajectories near invariant manifolds of slow motions. 
This is the basis of some model reduction approaches (see references in~\cite{Heiter2018}).
While there is no obvious natural definition of a SIM for general nonlinear systems, several competing eligible approaches have followed, most of which are based on Fenichel's perturbation theory for invariant manifolds~\cite{Fenichel1971,Fenichel1979}. 
The core idea is to introduce a parameter $\varepsilon$ into the generating vector field with a sufficiently smooth dependence in way that produces the original vector field for some absolutely small value $\varepsilon=\varepsilon_0\neq0$ and a system with annihilated slow modes for $\varepsilon=0$. 
Put in simple terms, Fenichel's theorem then ensures that for sufficiently small perturbations of $\varepsilon$ and some compact, normally hyperbolic submanifold $\mathcal{S}_0$ (with boundary) consisting of fixed points of the unperturbed system ($\varepsilon=0$), there exists a nearby diffeomorphic and normally hyperbolic submanifold $\mathcal{S}_{\varepsilon}$, overflowing invariant under the perturbed system and with the same attraction behavior in normal direction. 
This theorem is of great importance for the theory of dynamical systems and justifies the equivalent concepts of `center-like manifolds'~\cite{Sakamoto1990} or `slow (invariant) manifolds' (e.g.~\cite{Jones1995}), which are not unique in general but very close to each other~\cite{Tikhonov1952}. 
Important generalizations include the persistence of noncompact normally hyperbolic invariant manifolds (NHIMs) in Euclidean spaces~\cite{Sakamoto1990} and Riemannian manifolds of bounded geometry~\cite{Eldering2012} and of compact, possibly infinite-dimensional NHIMs in Banach spaces~\cite{Bates2008}. 
Similar concepts as for example `slow spectral submanifolds'~\cite{Haller2016} address the non-uniqueness by adding e.g.\ non-resonance constraints. \\
\noindent
In~\cite{Heiter2018} Heiter and Lebiedz promoted a differential geometric approach (originally rooted in the formulation of a variational problem~\cite{Lebiedz2004,Lebiedz2011a}) in order to avoid the artificially incorporated parameter $\varepsilon$ and provide an intrinsic geometric setting to obtain in the ideal case a local criterion. \\
\noindent
In the following, we will refer to the term slow invariant manifold (SIM) as NHIMs with purely stable normal direction. 
Classically, these SIMs are constructed in a unique way by using their persistence property. 
The idea is to asymptotically expand the unique SIM of a similar vector flow -- the so-called \emph{critical manifold} -- in the direction of perturbation. 
This results in a SIM that can be approximated by a power series in the perturbation parameter $\varepsilon$. 
While there is no general theorem providing uniqueness, the current article restricts to systems for which the resulting SIM is uniquely given under the assumption of analyticity. 

\section{Analytic Continuation}
\noindent
Consider a real analytical dynamical system given by the ordinary differential equation (ODE) initial value problems
\begin{equation}\label{ODE_iwp}
	\deriv{}{t} z = F(z), z(0)=z_0 \in \R^n
\end{equation}
whereas $F:U\to TU$ is an complete and (real) analytic vector field on a nonempty open subset $U\subset\R^n$, $TU$: tangent bundle.
This real analytic vector field can be continued to a complex analytic vector field on a complex open subset $\tilde U\subset\C^n$ with $U\subset\tilde U$. 
\begin{lemma}
	For every real analytic function $F:U \to \R^n$ there is an complex open set $\tilde U 
	\subset\C^n, U\subset\tilde U$ and a holomorphic function $\tilde F:\tilde U\to\C^n$ with
	$\tilde F|_U = F$ and $\tilde F$ is unique on $\tilde U$. 
\end{lemma}
\begin{proof}
	For an arbitrary point $p\in U$ denote the (poly-)radius of convergence of the local power series by $r(p)\in{(\R^+)}^n$. 
	Clearly, the corresponding complex power series (denoted by $\tilde F_p$) has the same radius of convergence and defines a holomorphic function on the corresponding polydisk
	\begin{equation*}
		P_p := \{z\in\C^n : |z_k-p_k|<r_k(p), k=1,\dots,n\}. 
	\end{equation*}
	As a consequence we obtain a complex open cover $\tilde U:=\bigcup_{p\in U} P_p$ of $U$ and a family of holomorphic functions $\tilde F_p$, each defined on the corresponding polydisk. 
	In addition, it holds $\tilde F_{p^1} \equiv \tilde F_{p^2}$ on $P_{p^1} \cap P_{p^2}$ for all $p^1,p^2\in\tilde U$ because 
	\begin{equation*}
		P_{p^1} \cap P_{p^2} \neq \emptyset \Rightarrow \exists p \in U \cap P_{p^1} \cap P_{p^2}
	\end{equation*}
	and by the principle of analytic continuation we have
	\begin{align*}
		\pderiv[k]{}{z} (\tilde F_{p^1}-\tilde F_{p^2})(p) = \pderiv[k]{}{z} (F-F)(p) = 0~\forall
		k\in\N^n \\
		\Rightarrow\quad \tilde F_{p^1}-\tilde F_{p^2} \equiv 0 \text{~on~} P_{p^1} \cap P_{p^2}. 
	\end{align*}
	Therefore, $\tilde F(z) := \tilde F_{p}(z)$ for $z\in P_p$ is well-defined. 
	By the same argument $\tilde F$ is unique on $\tilde U$. 
\end{proof}
\noindent
The next theorem provides existence and uniqueness results for holomorphic complex-time flows generated by holomorphic vector fields. 

\begin{theorem}[\cite{Ilyashenko2007}]
	For any holomorphic differential equation
	\begin{equation*}
		\deriv{}{t} \tilde z = \tilde F(\tilde z), \quad \tilde U\subset\C^n, \tilde F\in
		\mathcal{O}^n(\tilde U)
	\end{equation*}
	and every $\tilde z_0\in\tilde U$ there exists a sufficiently small polydisk $P(\tilde z_0,r)$ and $\varepsilon>0$ such that a solution to the corresponding initial value problem with $\tilde z(0) = \tilde z_0$ exists for all $t\in\C: |t|<\varepsilon$ and is unique on this polydisk. 
	Denote the corresponding flow by $\varphi^t(\tilde z_0)$ or $\exp(t \tilde F) \tilde z_0$ to emphasize the generating vector field. 
\end{theorem}
\noindent
Obviously, the real flow generated by the (real) vector field $F$ from above coincides with the restriction of complex flow to the real space, i.e.
\begin{equation*}
	\exp(t F)z_0 = \exp(t \tilde F) z_0 \quad \forall t\in\R:|t|<\varepsilon, z_0\in U\subset\R^n.
\end{equation*}
The complex analytic trajectories provide a well-defined expansion of the real ones. 
Therefore, it seems reasonable to expect complex manifestations of real space flow characteristics. 
While such insights are of general interest, this is especially the case for slow invariant manifolds as it might provide an explicit characterization of the so far implicitly given object. 

\section{Oscillations in imaginary time}
\noindent
Studying the geometric properties of the complex integral curves, which by definition form Riemann surfaces, we observe an oscillation in imaginary time direction that closely relates to spectral properties of the generating nonlinear vector field. 
Fast time scale dynamics correspond to high-frequency and slow time scales to low-frequency oscillations. \\
\noindent
In the following we focus on linear and nonlinear dissipative systems with one globally attracting fixed point often used for demonstration purposes in SIM research. 
\begin{example}[Linear system]
	Consider the linear dynamical system given by the differential equation $\dot z = Az$ with $A\in\R^{n\times n}$ being a real diagonalizable matrix with eigenvalues $\lambda_1,\dots,\lambda_n\in\R$ counted by their multiplicity and a corresponding basis of eigenvectors $v^1,\dots,v^n\in\R^n$. 
	For this model a SIM of dimension $j$ is the span of the $j$ slowest eigenvectors. 
	The imaginary time trajectory for $t\in\R$ and $z(0)=\sum_k \alpha_k v^k, \alpha_k \in\R$ is given by
	\begin{align*}
		z(\im t) =& \exp(\im tA) \sum_{k=1}^n \alpha_k v^k \\
		=& \sum_{k=1}^n \alpha_k \cos(\lambda_k t) v^k + \im \sum_{k=1}^n \alpha_k \sin(\lambda_k t)
		v^k. 
	\end{align*}
	This corresponds to the discrete Fourier spectrum
	\begin{equation*}
		\F_t [z^j(\im t)](\xi) = \sqrt{2\pi} \sum_{k=1}^n \alpha_k v_j^k\updelta(\xi-\lambda_k), 
	\end{equation*}
	where $\updelta$ denotes the Dirac delta distribution. 
	As long as $A$ is real diagonalizable this equation characterizes exponential growth rates in real time by frequencies in imaginary time. 
\end{example}

\begin{figure*}[!htbp]
	\includegraphics[width=0.48\textwidth]{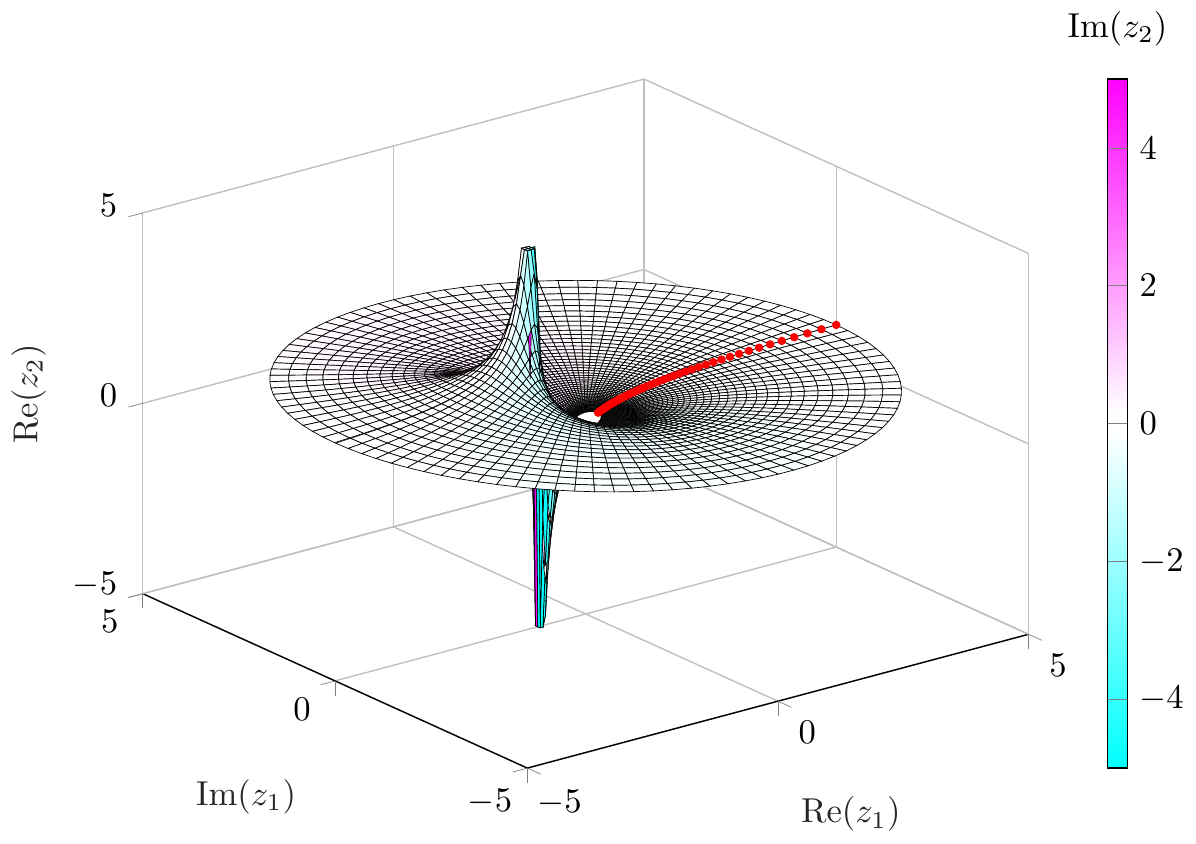}
	\includegraphics[width=0.48\textwidth]{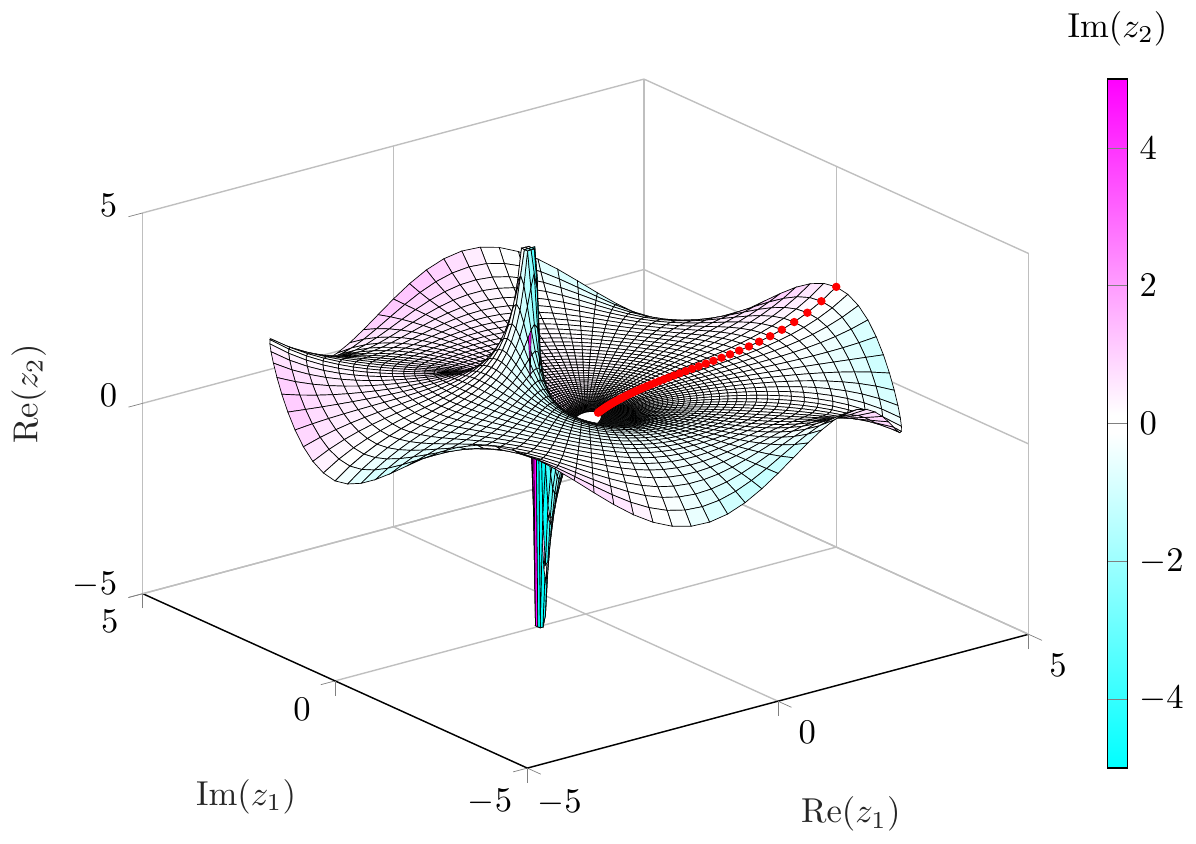}
	\caption{Davis-Skodje model: Complex time trajectories (Riemann surfaces) with initial point on (left) and off the SIM (right). The red line is the real time trajectory.}%
	\label{fig:ds_surf}
\end{figure*}

\begin{figure*}[!htbp]
	\includegraphics[width=0.48\textwidth]{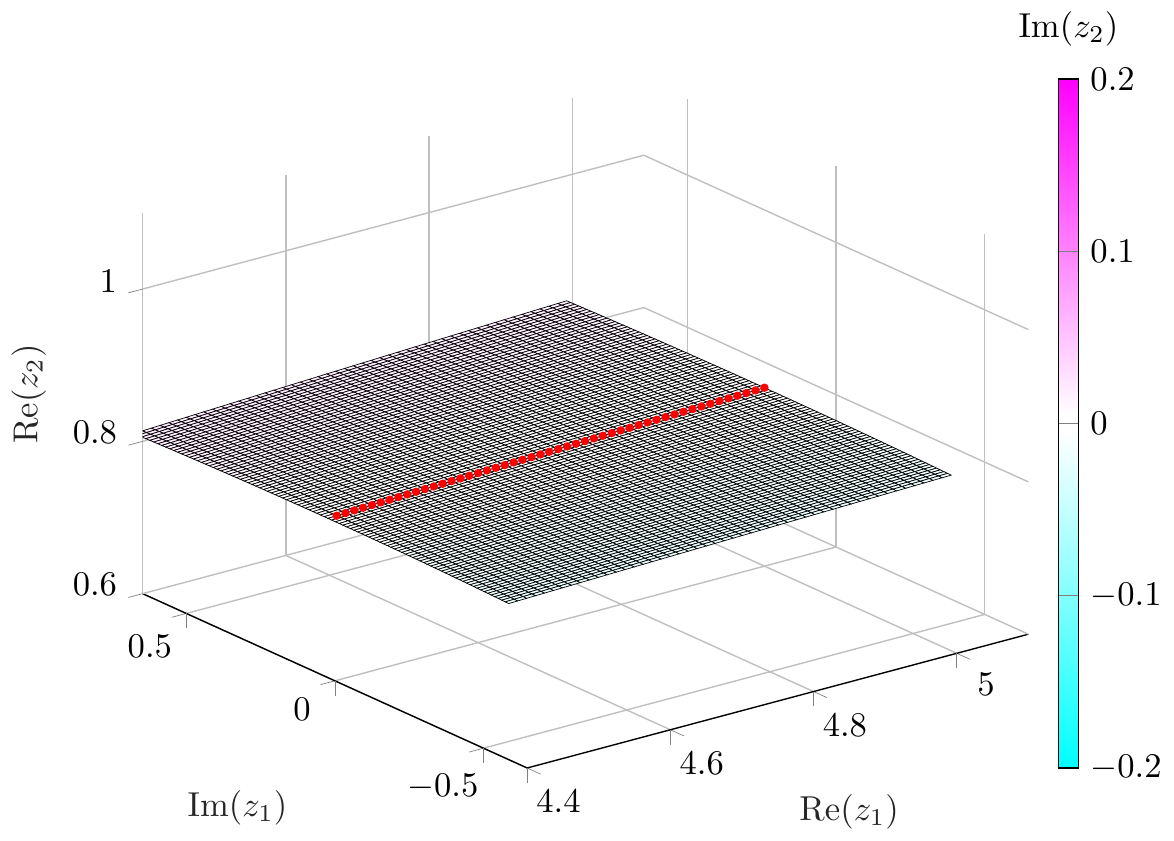}
	\includegraphics[width=0.48\textwidth]{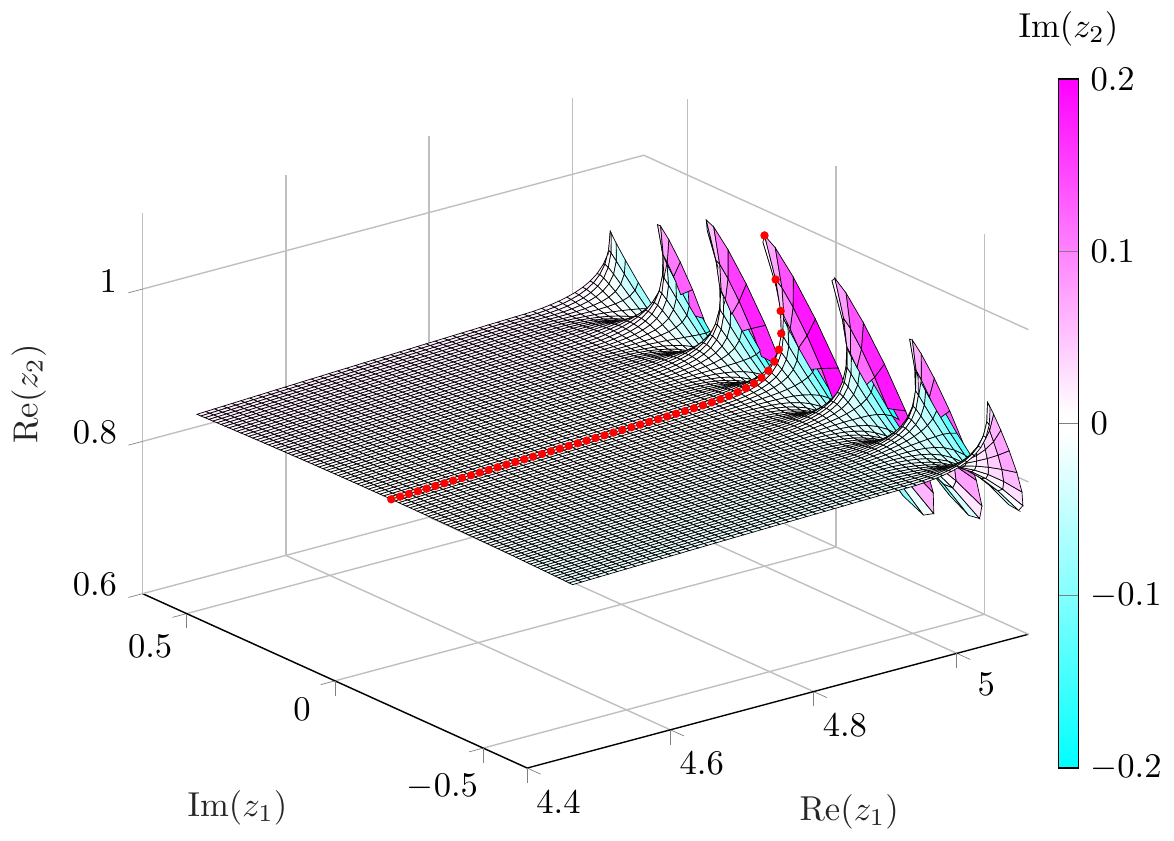}
	\caption{Michaelis-Menten model: Complex time trajectories (Riemann surfaces) with initial point on (left) and off the SIM (right). The red line is the real time trajectory.}%
	\label{fig:mm_surf}
\end{figure*}

\begin{example}[Davis-Skodje model]
	The Davis-Skodje dynamical system solves the differential equation
	\begin{align*}
		\dot z_1 &= -z_1 \\
		\dot z_2 &= -\gamma z_2 + \frac{(\gamma-1)z_1+\gamma z_1^2}{{(1+z_1)}^2}, 
	\end{align*}
	where $\gamma>1$ controls the time scale separation (spectral gap). This model has general solution 
	\begin{align*}
		z_1(t) &= c_1 \e^{-t} \\
		z_2(t) &= \frac{c_1\e^{-t}}{c_1\e^{-t}+1} + c_2 \e^{-\gamma t}. 
	\end{align*}
	An asymptotic expansion in $\varepsilon:=\frac{1}{\gamma}$ provides the SIM:
	\begin{equation*}
		S_\varepsilon=\left\{z\in\R^2: z_2 = \frac{z_1}{1+z_1} \right\}.
	\end{equation*}
	Figure~\ref{fig:ds_surf} reveals the oscillation in imaginary time direction. 
	Obviously, $z_1(\im t)$ is $2\pi$-periodic and $z_2(\im t)$ is a sum of periodic functions with period lengths $2\pi$ and $2\pi/\gamma$, respectively. 
	Using the Euclidean basis vectors $\mathbf{e}_1,\mathbf{e}_2$ and assuming that $c_1\in\R\setminus\{-1,0\}$ the Fourier transform can be written as
	\begin{align*}
		&\F_t [z(\im t)](\xi) = c_1 \updelta(\xi-1) \mathbf{e}_1 \\
		&\qquad+ \left( c_2 \updelta(\xi-\gamma) + \sum_{k=0}^\infty
		\frac{{(-1)}^k}{c_1^k} \updelta(\xi-k) \right) \mathbf{e}_2.
	\end{align*}
	Being on the slow invariant manifold ($c_2=0$) implies the absence of the high frequency in the second component. 
\end{example}

\begin{figure*}[!htbp]
	\centering
	\includegraphics[height=0.35\textwidth]{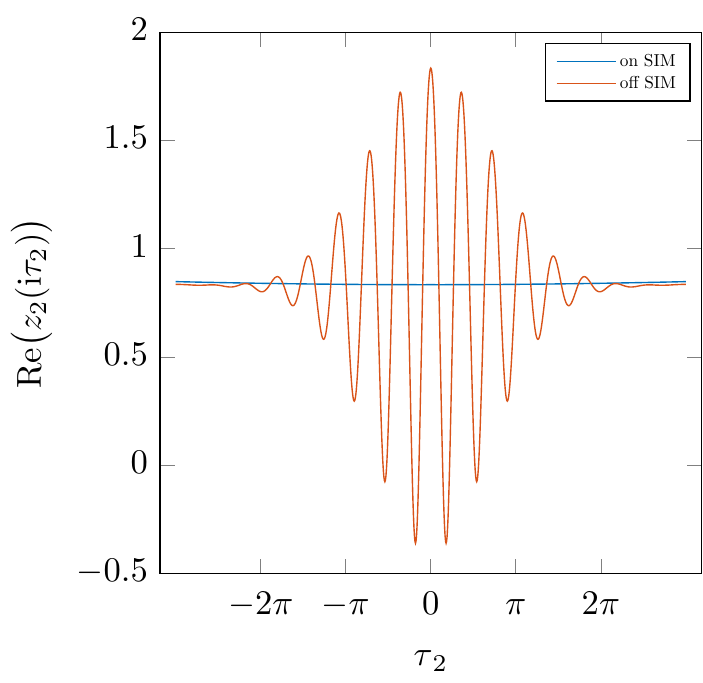}
	\hspace{2cm}
	\includegraphics[height=0.35\textwidth]{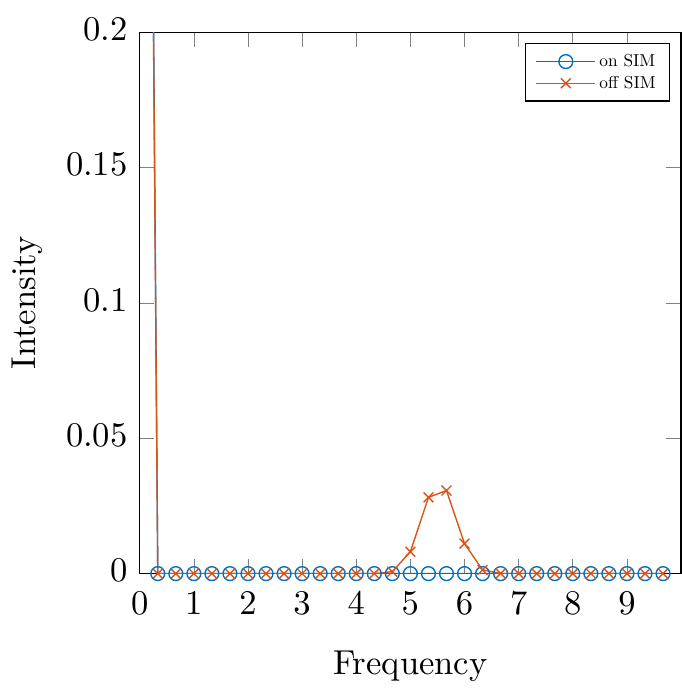}
	\caption{Michaelis-Menten model. Left: Real part of the imaginary time component function of $z_2$ for an initial point on (blue) and off (orange) the SIM\@. 
	Right: Corresponding Fourier spectrum (Fast Fourier Transform using MATLAB) with Dirac delta distribution at 0. 
	The spectrum of the trajectory with initial point off (orange) the SIM shows intensities in a range of continuous high frequencies.}%
	\label{fig:mm_freq}
\end{figure*}

\begin{example}[Michaelis-Menten model]
	\begin{align*}
		\dot z_1 &= \frac{1}{\gamma} \left( -z_1 +z_1 z_2 +\frac{z_2}{2} \right) \\
		\dot z_2 &= z_1 -z_1 z_2 +z_2. 
	\end{align*}
	Again a SIM can be found by asymptotic expansion: 
	\begin{align*}
		S_\varepsilon = \Bigg\{z\in\R^2: z_2=&\frac{z_1}{1+z_1} + \varepsilon \frac{z_1}{2(1+z_1)} \\
		&+ \varepsilon^2 \frac{z_1(1-\frac{5}{2}z_1)}{{2(1+z_1)}^7} + \dots\Bigg\}.
	\end{align*}
	This model is not explicitly integrable and differs in terms of spectral properties from the previous examples. 
	Figure~\ref{fig:mm_surf} shows a (decaying) oscillation in imaginary time. Again, this fast oscillation vanishes for SIM trajectories. \\
	\noindent
	In the numerically determined continuous Fourier spectrum this corresponds to the disappearance of the continuous part in the high-frequency domain (cf.~Figure~\ref{fig:mm_freq}). 
\end{example}

\section{Analytical considerations}
\noindent
The previous results suggest that there is a deep and fundamental connection between the imaginary time oscillation and the growth rate of trajectories in real phase space. 
Locally, this connection is obvious, when considering the first variational equation of problem (\ref{ODE_iwp}):
\begin{equation*}
	\frac{\diff}{\diff \tau_2} w(\im\tau_2) = \im J_{\tilde F}(\tilde z(\im\tau_2))\cdot w(\im\tau_2), 
	w(0) = w_0 \in T_{z_0}\R^n.
\end{equation*}
Exploiting the regularity of the vector field, one gets the approximation
\begin{equation*}
	w(\im\tau_2) \approx \exp(\im\tau_2 J_{\tilde F}(\tilde z(0))) \cdot w(0).
\end{equation*}
This serves as a first-order heuristic argument for imaginary time spectral oscillations.
More ambitious methods as the Magnus expansion~\cite{Magnus1954,Blanes2009} provide better approximations but require more information concerning the imaginary solution trajectory. \\
\noindent
In the following the issue is studied in some more depth using tools from harmonic analysis. 
Consider $z(\cdot)$ is entire in each component. 
A variation of the Schwartz-Paley-Wiener theorem~\cite{Hoermander1990} then implies
\begin{equation*}
	\|z(t)\|_1 \le C{(1+|t|)}^N \e^{\lambda|\Re(t)|}
\end{equation*}
for some constants $C,N,\lambda$ if and only if the support of the imaginary spectrum---in general a distribution---is compact and bounded: 
$\supp \mathcal{F}_t\left[ z_k(\im t) \right] \subset [-\lambda,\lambda] \quad\forall k$. \\
\noindent
Applied to SIM trajectories with entire component functions of exponential type~\cite{Stein2003} (i.e.\ with (less than) exponential growth) this provides a lower boundary of the imaginary time spectrum  w.r.t.\ close trajectories.
This is due to the (heuristic) argument that normal hyperbolicity implies an increased growth rate for trajectories perturbed from the SIM in normal direction. 
A spectrum of bounded support then corresponds to an absence of frequencies with absolute value above the upper bound. 
As a result high frequencies in imaginary time are a sufficient criterion for non-SIM trajectories. 
Therefore, the imaginary time spectrum is a potential indicator for SIM trajectories.

\section{Outlook}
\noindent
We propose that a complex-time view might be of general interest for the analysis of backbone structures in the phase space of real-analytic dynamical systems. 
In particular separatrices which separate qualitatively different topologies of orbit families or stability regions of fixed points.  
Recent work on vector fields generated by number theoretically significant functions like the
Riemann $\zeta$- and $\xi$-function suggests a role of separatrices for the potential locations of fixed points in phase space, especially for Newton flows. The work~\cite{Heitel2019a} analyzes characteristic properties of separatrices and their phase space behavior. 
The latter study is essentially based on real-time dynamical systems but in the outlook we speculate on an enlightening role of a complex-time extension giving rise to structural distinction of Riemann surfaces corresponding to separatrices compared to non-separatrices.

\section*{Acknowledgments}
We gratefully acknowledge the Klaus-Tschira-Stiftung for financial support. Furthermore, the authors thank Marcus Heitel, Johannes Poppe, and Marius M\"uller for discussions on this topic. 

\bibliography{literature.bib}

%apsrev4-2.bst 2019-01-14 (MD) hand-edited version of apsrev4-1.bst
%Control: key (0)
%Control: author (8) initials jnrlst
%Control: editor formatted (1) identically to author
%Control: production of article title (0) allowed
%Control: page (0) single
%Control: year (1) truncated
%Control: production of eprint (0) enabled
\begin{thebibliography}{20}%
\makeatletter
\providecommand \@ifxundefined [1]{%
 \@ifx{#1\undefined}
}%
\providecommand \@ifnum [1]{%
 \ifnum #1\expandafter \@firstoftwo
 \else \expandafter \@secondoftwo
 \fi
}%
\providecommand \@ifx [1]{%
 \ifx #1\expandafter \@firstoftwo
 \else \expandafter \@secondoftwo
 \fi
}%
\providecommand \natexlab [1]{#1}%
\providecommand \enquote  [1]{``#1''}%
\providecommand \bibnamefont  [1]{#1}%
\providecommand \bibfnamefont [1]{#1}%
\providecommand \citenamefont [1]{#1}%
\providecommand \href@noop [0]{\@secondoftwo}%
\providecommand \href [0]{\begingroup \@sanitize@url \@href}%
\providecommand \@href[1]{\@@startlink{#1}\@@href}%
\providecommand \@@href[1]{\endgroup#1\@@endlink}%
\providecommand \@sanitize@url [0]{\catcode `\\12\catcode `\$12\catcode
  `\&12\catcode `\#12\catcode `\^12\catcode `\_12\catcode `\%12\relax}%
\providecommand \@@startlink[1]{}%
\providecommand \@@endlink[0]{}%
\providecommand \url  [0]{\begingroup\@sanitize@url \@url }%
\providecommand \@url [1]{\endgroup\@href {#1}{\urlprefix }}%
\providecommand \urlprefix  [0]{URL }%
\providecommand \Eprint [0]{\href }%
\providecommand \doibase [0]{https://doi.org/}%
\providecommand \selectlanguage [0]{\@gobble}%
\providecommand \bibinfo  [0]{\@secondoftwo}%
\providecommand \bibfield  [0]{\@secondoftwo}%
\providecommand \translation [1]{[#1]}%
\providecommand \BibitemOpen [0]{}%
\providecommand \bibitemStop [0]{}%
\providecommand \bibitemNoStop [0]{.\EOS\space}%
\providecommand \EOS [0]{\spacefactor3000\relax}%
\providecommand \BibitemShut  [1]{\csname bibitem#1\endcsname}%
\let\auto@bib@innerbib\@empty
%</preamble>
\bibitem [{\citenamefont {Wick}(1954)}]{Wick1954}%
  \BibitemOpen
  \bibfield  {author} {\bibinfo {author} {\bibfnamefont {G.~C.}\ \bibnamefont
  {Wick}},\ }\bibfield  {title} {\bibinfo {title} {Properties of bethe-salpeter
  wave functions},\ }\href {https://doi.org/10.1103/PhysRev.96.1124} {\bibfield
   {journal} {\bibinfo  {journal} {Physical Review}\ }\textbf {\bibinfo
  {volume} {96}},\ \bibinfo {pages} {1124} (\bibinfo {year}
  {1954})}\BibitemShut {NoStop}%
\bibitem [{\citenamefont {Ochs}(2011)}]{Ochs2011}%
  \BibitemOpen
  \bibfield  {author} {\bibinfo {author} {\bibfnamefont {K.}~\bibnamefont
  {Ochs}},\ }\bibfield  {title} {\bibinfo {title} {A comprehensive analytical
  solution of the nonlinear pendulum},\ }\href
  {https://doi.org/10.1088/0143-0807/32/2/019} {\bibfield  {journal} {\bibinfo
  {journal} {European Journal of Physics}\ }\textbf {\bibinfo {volume} {32}},\
  \bibinfo {pages} {479} (\bibinfo {year} {2011})}\BibitemShut {NoStop}%
\bibitem [{\citenamefont {Romero}(2018)}]{Romero2018}%
  \BibitemOpen
  \bibfield  {author} {\bibinfo {author} {\bibfnamefont {R.~L.}\ \bibnamefont
  {Romero}},\ }\bibfield  {title} {\bibinfo {title} {Duality symmetries behind
  solutions of the classical simple pendulum},\ }\href
  {http://www.scielo.org.mx/scielo.php?script=sci_arttext&pid=S1870-35422018000200205&nrm=iso}
  {\bibfield  {journal} {\bibinfo  {journal} {Revista Mexicana de Fisica E}\
  }\textbf {\bibinfo {volume} {64}},\ \bibinfo {pages} {205} (\bibinfo {year}
  {2018})}\BibitemShut {NoStop}%
\bibitem [{\citenamefont {Heiter}\ and\ \citenamefont
  {Lebiedz}(2018)}]{Heiter2018}%
  \BibitemOpen
  \bibfield  {author} {\bibinfo {author} {\bibfnamefont {P.}~\bibnamefont
  {Heiter}}\ and\ \bibinfo {author} {\bibfnamefont {D.}~\bibnamefont
  {Lebiedz}},\ }\bibfield  {title} {\bibinfo {title} {Towards differential
  geometric characterization of slow invariant manifolds in extended phase
  space: Sectional curvature and flow invariance},\ }\href
  {https://doi.org/10.1137/16m1106353} {\bibfield  {journal} {\bibinfo
  {journal} {SIAM Journal on Applied Dynamical Systems}\ }\textbf {\bibinfo
  {volume} {17}},\ \bibinfo {pages} {732} (\bibinfo {year} {2018})}\BibitemShut
  {NoStop}%
\bibitem [{\citenamefont {Fenichel}(1971)}]{Fenichel1971}%
  \BibitemOpen
  \bibfield  {author} {\bibinfo {author} {\bibfnamefont {N.}~\bibnamefont
  {Fenichel}},\ }\bibfield  {title} {\bibinfo {title} {Persistence and
  smoothness of invariant manifolds for flows},\ }\href
  {https://doi.org/https://doi.org/10.1512/iumj.1971.21.21017} {\bibfield
  {journal} {\bibinfo  {journal} {Indiana University Mathematics Journal}\
  }\textbf {\bibinfo {volume} {21}},\ \bibinfo {pages} {193} (\bibinfo {year}
  {1971})}\BibitemShut {NoStop}%
\bibitem [{\citenamefont {Fenichel}(1979)}]{Fenichel1979}%
  \BibitemOpen
  \bibfield  {author} {\bibinfo {author} {\bibfnamefont {N.}~\bibnamefont
  {Fenichel}},\ }\bibfield  {title} {\bibinfo {title} {Geometric singular
  perturbation theory for ordinary differential equations},\ }\href
  {https://doi.org/10.1016/0022-0396(79)90152-9} {\bibfield  {journal}
  {\bibinfo  {journal} {Journal of Differential Equations}\ }\textbf {\bibinfo
  {volume} {31}},\ \bibinfo {pages} {53} (\bibinfo {year} {1979})}\BibitemShut
  {NoStop}%
\bibitem [{\citenamefont {Sakamoto}(1990)}]{Sakamoto1990}%
  \BibitemOpen
  \bibfield  {author} {\bibinfo {author} {\bibfnamefont {K.}~\bibnamefont
  {Sakamoto}},\ }\bibfield  {title} {\bibinfo {title} {Invariant manifolds in
  singular perturbation problems for ordinary differential equations},\ }\href
  {https://doi.org/10.1017/S0308210500031371} {\bibfield  {journal} {\bibinfo
  {journal} {Proceedings of the Royal Society of Edinburgh: Section A
  Mathematic}\ }\textbf {\bibinfo {volume} {116}},\ \bibinfo {pages} {45}
  (\bibinfo {year} {1990})}\BibitemShut {NoStop}%
\bibitem [{\citenamefont {Jones}(1995)}]{Jones1995}%
  \BibitemOpen
  \bibfield  {author} {\bibinfo {author} {\bibfnamefont {C.~K. R.~T.}\
  \bibnamefont {Jones}},\ }\bibfield  {title} {\bibinfo {title} {Geometric
  singular perturbation theory},\ }in\ \href
  {https://doi.org/10.1007/bfb0095239} {\emph {\bibinfo {booktitle} {Dynamical
  Systems}}},\ \bibinfo {series and number} {\bibinfo {series} {Lecture Notes
  in Mathematics}\ No.\ \bibinfo {number} {1609}},\ \bibinfo {editor} {edited
  by\ \bibinfo {editor} {\bibfnamefont {R.}~\bibnamefont {Johnson}}}\ (\bibinfo
   {publisher} {Springer},\ \bibinfo {year} {1995})\ Chap.~\bibinfo {chapter}
  {2}, pp.\ \bibinfo {pages} {44--118}\BibitemShut {NoStop}%
\bibitem [{\citenamefont {Tikhonov}(1952)}]{Tikhonov1952}%
  \BibitemOpen
  \bibfield  {author} {\bibinfo {author} {\bibfnamefont {A.~N.}\ \bibnamefont
  {Tikhonov}},\ }\bibfield  {title} {\bibinfo {title} {Systems of differential
  equations containing small parameters in the derivatives},\ }\href@noop {}
  {\bibfield  {journal} {\bibinfo  {journal} {Matematicheskii sbornik}\
  }\textbf {\bibinfo {volume} {73}},\ \bibinfo {pages} {575} (\bibinfo {year}
  {1952})}\BibitemShut {NoStop}%
\bibitem [{\citenamefont {Eldering}(2012)}]{Eldering2012}%
  \BibitemOpen
  \bibfield  {author} {\bibinfo {author} {\bibfnamefont {J.}~\bibnamefont
  {Eldering}},\ }\bibfield  {title} {\bibinfo {title} {Persistence of
  noncompact normally hyperbolic invariant manifolds in bounded geometry},\
  }\href {https://doi.org/https://doi.org/10.1016/j.crma.2012.06.009}
  {\bibfield  {journal} {\bibinfo  {journal} {Comptes Rendus Mathematique}\
  }\textbf {\bibinfo {volume} {350}},\ \bibinfo {pages} {617} (\bibinfo {year}
  {2012})}\BibitemShut {NoStop}%
\bibitem [{\citenamefont {Bates}\ \emph {et~al.}(2008)\citenamefont {Bates},
  \citenamefont {Lu},\ and\ \citenamefont {Zeng}}]{Bates2008}%
  \BibitemOpen
  \bibfield  {author} {\bibinfo {author} {\bibfnamefont {P.~W.}\ \bibnamefont
  {Bates}}, \bibinfo {author} {\bibfnamefont {K.}~\bibnamefont {Lu}},\ and\
  \bibinfo {author} {\bibfnamefont {C.}~\bibnamefont {Zeng}},\ }\bibfield
  {title} {\bibinfo {title} {Approximately invariant manifolds and global
  dynamics of spike states},\ }\href
  {https://doi.org/10.1007/s00222-008-0141-y} {\bibfield  {journal} {\bibinfo
  {journal} {Inventiones mathematicae}\ }\textbf {\bibinfo {volume} {174}},\
  \bibinfo {pages} {355} (\bibinfo {year} {2008})}\BibitemShut {NoStop}%
\bibitem [{\citenamefont {Haller}\ and\ \citenamefont
  {Ponsioen}(2016)}]{Haller2016}%
  \BibitemOpen
  \bibfield  {author} {\bibinfo {author} {\bibfnamefont {G.}~\bibnamefont
  {Haller}}\ and\ \bibinfo {author} {\bibfnamefont {S.}~\bibnamefont
  {Ponsioen}},\ }\bibfield  {title} {\bibinfo {title} {Nonlinear normal modes
  and spectral submanifolds: existence, uniqueness and use in model
  reduction},\ }\href {https://doi.org/10.1007/s11071-016-2974-z} {\bibfield
  {journal} {\bibinfo  {journal} {Nonlinear Dynamics}\ }\textbf {\bibinfo
  {volume} {86}},\ \bibinfo {pages} {1493} (\bibinfo {year}
  {2016})}\BibitemShut {NoStop}%
\bibitem [{\citenamefont {Lebiedz}(2004)}]{Lebiedz2004}%
  \BibitemOpen
  \bibfield  {author} {\bibinfo {author} {\bibfnamefont {D.}~\bibnamefont
  {Lebiedz}},\ }\bibfield  {title} {\bibinfo {title} {Computing minimal entropy
  production trajectories: An approach to model reduction in chemical
  kinetics},\ }\href {https://doi.org/10.1063/1.1652428} {\bibfield  {journal}
  {\bibinfo  {journal} {Journal of Chemical Physics}\ }\textbf {\bibinfo
  {volume} {120}},\ \bibinfo {pages} {6890} (\bibinfo {year}
  {2004})}\BibitemShut {NoStop}%
\bibitem [{\citenamefont {Lebiedz}\ \emph {et~al.}(2011)\citenamefont
  {Lebiedz}, \citenamefont {Siehr},\ and\ \citenamefont
  {Unger}}]{Lebiedz2011a}%
  \BibitemOpen
  \bibfield  {author} {\bibinfo {author} {\bibfnamefont {D.}~\bibnamefont
  {Lebiedz}}, \bibinfo {author} {\bibfnamefont {J.}~\bibnamefont {Siehr}},\
  and\ \bibinfo {author} {\bibfnamefont {J.}~\bibnamefont {Unger}},\ }\bibfield
   {title} {\bibinfo {title} {A variational principle for computing slow
  invariant manifolds in dissipative dynamical systems},\ }\href
  {https://doi.org/10.1137/100790318} {\bibfield  {journal} {\bibinfo
  {journal} {SIAM Journal on Scientific Computing}\ }\textbf {\bibinfo {volume}
  {33}},\ \bibinfo {pages} {703} (\bibinfo {year} {2011})}\BibitemShut
  {NoStop}%
\bibitem [{\citenamefont {Ilyashenko}(2007)}]{Ilyashenko2007}%
  \BibitemOpen
  \bibfield  {author} {\bibinfo {author} {\bibfnamefont {Y.}~\bibnamefont
  {Ilyashenko}},\ }\href {https://doi.org/10.1090/gsm/086} {\emph {\bibinfo
  {title} {Lectures on Analytic Differential Equations}}}\ (\bibinfo
  {publisher} {American Mathematical Society},\ \bibinfo {year}
  {2007})\BibitemShut {NoStop}%
\bibitem [{\citenamefont {Magnus}(1954)}]{Magnus1954}%
  \BibitemOpen
  \bibfield  {author} {\bibinfo {author} {\bibfnamefont {W.}~\bibnamefont
  {Magnus}},\ }\bibfield  {title} {\bibinfo {title} {On the exponential
  solution of differential equations for a linear operator},\ }\href
  {https://doi.org/10.1002/cpa.3160070404} {\bibfield  {journal} {\bibinfo
  {journal} {Commications in Pure and Applied Mathematics}\ }\textbf {\bibinfo
  {volume} {VII}},\ \bibinfo {pages} {649} (\bibinfo {year}
  {1954})}\BibitemShut {NoStop}%
\bibitem [{\citenamefont {Blanes}\ \emph {et~al.}(2009)\citenamefont {Blanes},
  \citenamefont {Casas}, \citenamefont {Oteo},\ and\ \citenamefont
  {Ros}}]{Blanes2009}%
  \BibitemOpen
  \bibfield  {author} {\bibinfo {author} {\bibfnamefont {S.}~\bibnamefont
  {Blanes}}, \bibinfo {author} {\bibfnamefont {F.}~\bibnamefont {Casas}},
  \bibinfo {author} {\bibfnamefont {J.}~\bibnamefont {Oteo}},\ and\ \bibinfo
  {author} {\bibfnamefont {J.}~\bibnamefont {Ros}},\ }\bibfield  {title}
  {\bibinfo {title} {The magnus expansion and some of its applications},\
  }\href {https://doi.org/10.1016/j.physrep.2008.11.001} {\bibfield  {journal}
  {\bibinfo  {journal} {Physics Reports}\ }\textbf {\bibinfo {volume} {470}},\
  \bibinfo {pages} {151} (\bibinfo {year} {2009})}\BibitemShut {NoStop}%
\bibitem [{\citenamefont {H{\"o}rmander}(1990)}]{Hoermander1990}%
  \BibitemOpen
  \bibfield  {author} {\bibinfo {author} {\bibfnamefont {L.}~\bibnamefont
  {H{\"o}rmander}},\ }\href {https://books.google.de/books?id=UpdaAAAAYAAJ}
  {\emph {\bibinfo {title} {The analysis of linear partial differential
  operators: Distribution theory and Fourier analysis}}},\ Springer Study
  Edition\ (\bibinfo  {publisher} {Springer-Verlag},\ \bibinfo {year}
  {1990})\BibitemShut {NoStop}%
\bibitem [{\citenamefont {Stein}\ and\ \citenamefont
  {Sharkachi}(2003)}]{Stein2003}%
  \BibitemOpen
  \bibfield  {author} {\bibinfo {author} {\bibfnamefont {E.~M.}\ \bibnamefont
  {Stein}}\ and\ \bibinfo {author} {\bibfnamefont {R.}~\bibnamefont
  {Sharkachi}},\ }\href
  {https://www.ebook.de/de/product/3683590/elias_m_stein_rami_shakarchi_complex_analysis.html}
  {\emph {\bibinfo {title} {Complex Analysis}}}\ (\bibinfo  {publisher}
  {Princeton University Press},\ \bibinfo {year} {2003})\BibitemShut {NoStop}%
\bibitem [{\citenamefont {Heitel}\ and\ \citenamefont
  {Lebiedz}(2019)}]{Heitel2019a}%
  \BibitemOpen
  \bibfield  {author} {\bibinfo {author} {\bibfnamefont {M.}~\bibnamefont
  {Heitel}}\ and\ \bibinfo {author} {\bibfnamefont {D.}~\bibnamefont
  {Lebiedz}},\ }\href@noop {} {\bibinfo {title} {On analytical and topological
  properties of separatrices in 1-d holomorphic dynamical systems and
  complex-time newton flows}} (\bibinfo {year} {2019}),\ \Eprint
  {https://arxiv.org/abs/1911.10963} {arXiv:1911.10963 [math.DS]} \BibitemShut
  {NoStop}%
\end{thebibliography}%

\end{document}